\documentclass[11pt]{article}
\usepackage{amsmath,amsthm}
\usepackage{amssymb,mathrsfs}
\usepackage{bm,mathtools}
\usepackage{graphics}
\usepackage{tikz}
\usepackage{float}
\usepackage{booktabs}
\usepackage{array}	
\usetikzlibrary{calc}
\usetikzlibrary{shapes.geometric}
\usepackage{xcolor}

\usepackage{geometry}
\usepackage[colorlinks=true,
linkcolor=blue,citecolor=blue,pagebackref,
urlcolor=blue]{hyperref}
\usepackage{booktabs}
\usepackage{array}

\makeatletter
\def\@seccntDot{.}
\def\@seccntformat#1{\csname the#1\endcsname\@seccntDot\hskip 0.5em}
\renewcommand\section{\@startsection{section}{1}{\z@}%
{18\p@ \@plus 6\p@ \@minus 3\p@}%
{9\p@ \@plus 6\p@ \@minus 3\p@}%
{\large\bfseries\boldmath}}
\renewcommand\subsection{\@startsection{subsection}{2}{\z@}%
{12\p@ \@plus 6\p@ \@minus 3\p@}%
{3\p@ \@plus 6\p@ \@minus 3\p@}%
{\bfseries\boldmath}}
\renewcommand\subsubsection{\@startsection{subsubsection}{3}{\z@}%
{12\p@ \@plus 6\p@ \@minus 3\p@}%
{\p@}%
{\bfseries\boldmath}}
\makeatother

\usepackage{microtype}
\usepackage{float}

\theoremstyle{plain}
\newtheorem{theorem}{Theorem}
\newtheorem{lemma}{Lemma}

\newtheorem{conj}{Conjecture}

\theoremstyle{definition}
\newtheorem{definition}{Definition}[section]
\newtheorem{remark}{Remark}

\newtheorem{claim}{Claim}

\numberwithin{equation}{section}
\allowdisplaybreaks
\title{Two conjectures on vertex-disjoint rainbow triangles}

\author{Xu Liu\footnote{College of Computer Science, Nankai University, Tianjin 300350, P.R. China.},~~~Bo Ning\footnote{Corresponding author. College of Computer Science, Nankai University, Tianjin 300350, P.R. China.
E-mail: \texttt{bo.ning@nankai.edu.cn} (B. Ning). Partially supported by the National Nature Science
Foundation of China (No. 12371350) and Fundamental Research Funds for the Central Universities, Nankai University (No. 63243151).},~~~Yuting Tian\footnote{College of Cryptology and Cyber Science, Nankai University, Tianjin 300350, P.R. China.}}

\date{}

\begin{document}
\maketitle	

\begin{abstract}
In 1963, Dirac proved that every $n$-vertex graph has $k$ vertex-disjoint triangles if $n\geq 3k$ and minimum degree $\delta(G)\geq \frac{n+k}{2}$.
The base case $n=3k$ can be reduced to the Corr\'adi-Hajn\'al Theorem.
Towards a rainbow version of Dirac's Theorem, Hu, Li, and Yang conjectured that for all positive integers $n$ and $k$ with $n\geq 3k$, every edge-colored graph $G$ of order $n$ with $\delta^c(G)\geq \frac{n+k}{2}$ contains $k$ vertex-disjoint rainbow triangles. 
In another direction, Wu et al. conjectured an exact formula for anti-Ramsey number $ar(n,kC_3)$, generalizing the earlier work of Erd\H{o}s, S\'os and Simonovits. The conjecture of Hu, Li, and Yang was confirmed for the cases $k=1$ and $k=2$. However, Lo and Williams disproved the conjecture when $n\leq \frac{17k}{5}.$ It is therefore natural to ask whether the conjecture holds for $n=\Omega(k)$.

In this paper, we confirm this by showing that the Hu-Li-Yang conjecture holds when $n\ge 42.5k+48$. We disprove the conjecture of Wu et al. and propose a modified conjecture.
This conjecture is motivated by previous works due to Allen, B\"{o}ttcher, Hladk\'{y}, and Piguet on Tur\'an number of vertex-disjoint triangles.
\end{abstract}

\section{Introduction}\label{sec1}

An edge-coloring of a graph $G$ is a mapping $f:E(G)\rightarrow \mathbb{N}$, where $\mathbb{N}$ is the set of natural numbers.  An edge-colored subgraph $H$ of $G$ is called rainbow if the edges of $H$ are assigned pairwise different colors.
Given a positive integer $n$ and a graph $H$, the \emph{anti-Ramsey number} $ar(n,H)$ is the maximum number of colors in an edge-coloring of $K_n$ such that it contains no rainbow $H$ as a subgraph.  For a vertex $v\in V(G)$, the 
\emph{color degree} of $v$, denoted by $d^c(v)$, is defined to be the number of different colors which are assigned to all edges incident to $v$. The \emph{minimum color degree}, denoted by $\delta^c(G)$, is defined to be $\min\{d^c(v):v\in V(G)\}$. The \emph{color number} of $G$, denoted by $c(G)$, is defined to be the size of $\{c(e):e\in E(G)\}$.
The \emph{Tur\'an number} $ex(n, H)$ is the maximum number of edges of a graph on $n$ vertices that contains no subgraph isomorphic to $H$.  

The main purpose of this paper is to study the following two conjectures on vertex-disjoint rainbow triangles.  In fact, we disprove the first conjecture and confirm the second conjecture when $n=\Omega(k)$.

\begin{conj}[\cite{WZLX2023}]\label{Conj:1}
\begin{align*}
 &ar(n,kC_3)= \max \left\{
\binom{3k-1}{2}+n-3k+1,\
\left\lfloor\frac{(n-k+2)^2}{4}\right\rfloor+(k-2)(n-k+2)+\binom{k-2}{2}+1
	\right\}   
\end{align*}
for all $n \geq 3k.$
\end{conj}

\begin{conj}[\cite{HLY20}]\label{Conj:2}
For all positive integers $n$ and $m$ with $n\ge 3k$, every edge-colored graph $G$ of order $n$ with $\delta^c(G)\ge (n+k)/2$ contains $m$ vertex-disjoint rainbow triangles.
\end{conj}

The first conjecture is about the anti-Ramsey number of vertex-disjoint triangles. The concept of anti-Ramsey number of a graph $H$ was introduced by Erd\H{o}s, S\'os, and Simonovits \cite{ESS65}. Although it is named anti-Ramsey numbers, it indeed is closely related Tur\'an problems than Ramsey Theory. For example, for general graphs $H$, we have $\min\{ex(n,H-e):e\in E(H)\}+1\leq ar(n,H)\leq ex(n,H)$. Erd\H{o}s et al. \cite{ESS65} once conjectured that $ar(n,C_k)= (\frac{k-2}{2}+\frac{1}{k-1})n + O(1)$ and proved that $ar(n,C_3)=n-1$. After several efforts (see \cite{A83,JW03}), this problem was finally solved in \cite{MN05}. Since the paper \cite{ESS65}, there have been many works that study anti-Ramsey numbers of graphs. We refer the reader to the survey \cite{FMO10} and the results therein. 

We define $kC_3$ as the vertex-disjoint union of $k$ triangles. Yuan and Zhang \cite{YZ19} determined the exact values of $ar(n,kC_3)$ when $n$ is sufficiently large using Simonovits' method. Wu et al. \cite{WZLX2023} improved the result in \cite{YZ19} by
proving $ar(n,kC_3)=\lfloor\frac{(n-k+1)^2}{4}\rfloor+(k-2)(n-k+2)+\binom{k-2}{2}+1$ for all $n\geq 2k^2-k+2.$ Furthermore, for the base case, they \cite{WZLX2023} proved $ar(3k,kC_3)=\binom{3k-1}{2}+1$, and $ar(n, kC_3) \leq \frac{(n-k+2)^2}{4} +(k -2)(n-k+2) + \binom{k-2}{2} +(k-1)^2- \frac{n-3k}{2}+1$ for all $3k\leq n \leq 2k^2-k+2$. Based on these results, they proposed Conjecture
\ref{Conj:1}. Only when we are writing this manuscript, we find that Lu, Luo, and Ma \cite{LLM25+} have very recently proved that for any two integers $n,k\geq 2$, and $n\geq 15k+27$, we have
$ar(n,kC_3)=\lfloor\frac{(n-k+1)^2}{4}\rfloor+(k-2)(n-k+2)+\binom{k-2}{2}+1$.
However, as shown in Section \ref{Sec:2}, we shall show that Conjecture \ref{Conj:1} is false. Very interestingly for us, the construction is motivated by the work of Allen, B\"{o}ttcher, Hladk\'{y}, and Piguet \cite{ABHP15} on Tur\'an number of vertex-disjoint triangles.

Our first contribution to this paper is as follows.
\begin{theorem}\label{thm:1}
Conjecture \ref{Conj:1} is false.
\end{theorem}
\begin{remark}
For all integers $n \geq 3k $ where $0.24n\lesssim k\lesssim 0.3n$, there exist two edge-colorings of $K_n$: $G_2(n,k)$ and $G_3(n,k)$ (for details, see the next section and Fig. 1), which contain no rainbow $kC_3$, we can see that $c(G_2(n,k)$ and $c(G_3(n,k)$ are both larger than the conjectured values of $ar(n,kC_3)$, thus disproving Conjecture \ref{Conj:1}.    
\end{remark}

The second conjecture has many motivations. A well-known fact is that a minimum degree of at least $\frac{n+1}{2}$ ensures that every graph on $n$ vertices contains a triangle. It is natural to ask an edge-colored version of this observation. In 2007 (and formally in 2012), Li and Wang \cite{LW12} conjectured that every edge-colored graph has a rainbow triangle if $\delta^c(G)\geq \frac{n+1}{2}$. This conjecture was confirmed by H. Li \cite{L13} and independently in \cite{LNXZ14} . 
\begin{theorem}[\cite{L13}]\label{Thm:Li}
 Let $G$ be an edge-colored graph on $n\geq 3$ vertices. If $\delta^c(G)\geq \frac{n+1}{2}$ then $G$ contains a rainbow triangle.   
\end{theorem}

\begin{theorem}[\cite{LNXZ14}]
 Let $G$ be an edge-colored graph on $n\geq 5$ vertices. If $\delta^c(G)\geq \frac{n}{2}$, then $G$ contains a rainbow triangle, unless $G$ is a properly colored $K_{\frac{n}{2},\frac{n}{2}}$ where $n$ is even.
\end{theorem}

There are several extensions; for example, Czygrinow, Molla, Nagle, and
Oursler [7] proved that H. Li's condition in Theorem \ref{Thm:Li} ensures a rainbow $\ell$‐cycle $C_{\ell}$
whenever $n>432\ell$, which is sharp for a fixed odd integer $\ell\geq 3$ when $n$ is sufficiently large. For more related results, see \cite{CKRY16,LNSZ24,CN25}. 

In 2020, Hu, H. Li, and Yang \cite{HLY20} proved that every edge-colored graph on $n\geq 20$ vertices has two vertex-disjoint rainbow triangles. This result was slighted improved to $n\geq 6$ in \cite{CLN22}. So, Conjecture \ref{Conj:2} is true for $k=1,2$. Dirac \cite{D63} proved that every graph $G$ on $n\geq 3k$ vertices has $k$ vertex-disjoint triangles if $\delta(G)\geq \frac{n+k}{2}$. The base case $n=3k$ is equivalent to the famous Corr\'adi-Hajn\'al Theorem \cite{CH63} which states that every $n$-vertex graph with $n=3k$ and minimum degree $\delta(G)\geq 2k$ has a triangle-factor. 
So, if it were true, Conjecture \ref{Conj:2} can be seen as a rainbow version of Dirac's theorem. However, Lo and Williams \cite{LW-Arxiv-24} give a construction showing that this conjecture is false when $n\leq \frac{17k}{5}.$ Thus, it is natural to ask whether Conjecture \ref{Conj:2} holds for $n=\Omega(k)$. 
Our second result confirms this.
\begin{theorem}
 Conjecture \ref{Conj:2} is true for $n\ge 42.5k+48$.   
\end{theorem}

\begin{remark}
When we study $ar(n,kC_3)$, we need $kC_3$ to be rainbow, that is, the edges of all these vertex-disjoint triangles are assigned pairwise different colors. However, when we consider Conjecture \ref{Conj:2}, ``vertex-disjoint rainbow triangles" means that each vertex-disjoint triangle is rainbow, but two different rainbow triangles can have the same color.   
\end{remark}

\section{The disproof of Conjecture \ref{Conj:1}} \label{Sec:2}
We first introduce four classes of extremal graphs.
\begin{definition}
Let $n$ and $k$ be non-negative integers with $n \geq 3k$. We define four edge-colorings of $K_n$ as follows.
\end{definition}
\begin{itemize}
		\item $G_1(n,k)$: Let $V(G_1(n,k))=X\cup Y_1 \cup Y_2$, where $\arrowvert X\arrowvert=k-2$, $\arrowvert Y_1\arrowvert=\lfloor\frac{n-k+2}{2}\rfloor$, and $\arrowvert Y_2\arrowvert=\lceil\frac{n-k+2}{2}\rceil$.
		We color all edges in $K_{|X|} \vee K_{|Y_1|,|Y_2|}$ with distinct colors, and color $E(K_n)\setminus E(K_{|X|} \vee K_{|Y_1|,|Y_2|})$ with another new color, depicted in Figure \ref{fig1}, so
		\[c(G_1(n,k)=\left(\begin{matrix} k-2 \\ 2\end{matrix}\right)+(k-2)(n-k+2)+\left\lfloor\frac{(n-k+2)^2}{4}\right\rfloor+1.\]
		\begin{figure}[H]
	\centering
	\begin{tikzpicture}
	
	\draw[thick, rounded corners=10pt] (-1.5,-2.5) rectangle (5.5,2.5);
	
	\draw (0,0) ellipse (1cm and 2cm);
	\node at (0,2.2) {$Y_1$};
	\draw[fill=white] (0,1.5) circle (2pt) ;
	\draw[fill=white] (0,0) circle (2pt)  ;
	\draw[fill=white] (0,-1.5) circle (2pt)  ;
	\node at (0,-0.75) {$\vdots$};
	
	\draw (4,0) ellipse (1cm and 2cm);
	\node at (4,2.2) {$Y_2$};
	\draw[fill=white] (4,1.5) circle (2pt) ;
	\draw[fill=white] (4,0) circle (2pt) ;
	\draw[fill=white] (4,-1.5) circle (2pt) ;
	\node at (4,-0.75) {$\vdots$};
	
	\draw[red, thick] (0,1.5) -- (4,1.5);
	\draw[blue, thick] (0,1.5) -- (4,0);
	\draw[green, thick] (0,1.5) -- (4,-1.5);
	
	\draw[orange, thick] (0,0) -- (4,1.5);
	\draw[purple, thick] (0,0) -- (4,0);
	\draw[cyan, thick] (0,0) -- (4,-1.5);
	
	\draw[magenta, thick] (0,-1.5) -- (4,1.5);
	\draw[brown, thick] (0,-1.5) -- (4,0);
	\draw[teal, thick] (0,-1.5) -- (4,-1.5);
	
	\filldraw[fill=gray!30, draw=black, thick] (2,-4.2) ellipse (3cm and 1cm);
	\node at (2,-5.5) {$X$};
	\node at (2,-4.2) {$rb K_{k-2}$};
	\draw[red, line width=3pt] (1.7,-2.5) -- (1.7,-3.2);
	\draw[orange, line width=3pt] (1.8,-2.5) -- (1.8,-3.2);
	\draw[yellow, line width=3pt] (1.9,-2.5) -- (1.9,-3.2);
	\draw[green, line width=3pt] (2.0,-2.5) -- (2.0,-3.2);
	\draw[blue, line width=3pt] (2.1,-2.5) -- (2.1,-3.2);
	\draw[cyan, line width=3pt] (2.2,-2.5) -- (2.2,-3.2);
	\draw[magenta, line width=3pt] (2.3,-2.5) -- (2.3,-3.2);
	
	\node[font=\small] at (2.6, -2.85) {rb};
	
	\end{tikzpicture}
	
		\caption{$G_1(n,k)$}
		\label{fig1}
	\end{figure}
    
		\item $G_2(n.k)$: The second-class of extremal graphs is defined only for $k<\frac{n+7}{4}$. Let $V_2(G_2(n,k))$$=X \cup Y_1 \cup Y_2$ with $\arrowvert X\arrowvert=2k-3$, $\arrowvert Y_1\arrowvert=\lfloor \frac{n}{2}\rfloor$, and $Y_2=\lceil\frac{n}{2}-2k+3\rceil$ (or $\arrowvert Y_1\arrowvert=\lceil \frac{n}{2}\rceil$, and $Y_2=\lfloor\frac{n}{2}-2k+3 \rfloor$ ). 
		We color all edges in $|Y_1|K_1 \vee (K_{|X|} \cup |Y_2|K_1)$ with distinct colors, and color $E(K_n)\setminus E(|Y_1|K_1 \vee (K_{|X|} \cup |Y_2|K_1))$ with another new color, depicted in Figure \ref{fig2}, so
		\[c(G_2(n,k)=\left(\begin{matrix} 2k-3 \\ 2\end{matrix}\right)+\left\lfloor\frac{n^2}{4}\right\rfloor+1.\]

        \begin{figure}
		\centering
		\begin{tikzpicture}[
			dot/.style={circle, draw, fill=white, inner sep=0pt, minimum size=3pt},
			myellipse/.style={draw, ellipse, minimum width=1.5cm, minimum height=3cm, thick}
			]
			
			\filldraw[fill=gray!30, draw=black, thick] (0,0) ellipse (0.9cm and 1.5cm);
			\node at (0,2.0) {$X$};
			\node[dot] (x1) at (0,1.0) {};
			\node[dot] (x2) at (0,0) {};
			\node[dot] (x3) at (0,-1.0) {};
			\node at (0,-0.5) {$\vdots$};
			\node at (0, 0.5) {$ rb K_{2k-3}$};
			
			\draw[thick] (3,0) ellipse (0.75cm and 1.5cm);
			\node at (3,2.0) {$Y_1$};
			\node[dot] (y1) at (3,1.0) {};
			\node[dot] (y2) at (3,0) {};
			\node[dot] (y3) at (3,-1.0) {};
			\node at (3,-0.5) {$\vdots$};
			
			\draw[thick] (6,0) ellipse (0.75cm and 1.5cm);
			\node at (6,2.0) {$Y_2$};
			\node[dot] (z1) at (6,1.0) {};
			\node[dot] (z2) at (6,0) {};
			\node[dot] (z3) at (6,-1.0) {};
			\node at (6,-0.5) {$\vdots$};
			
			\draw[red, thick] (x1) -- (y1);
			\draw[blue, thick] (x1) -- (y2);
			\draw[green, thick] (x1) -- (y3);
			\draw[orange, thick] (x2) -- (y1);
			\draw[purple, thick] (x2) -- (y2);
			\draw[cyan, thick] (x2) -- (y3);
			\draw[magenta, thick] (x3) -- (y1);
			\draw[brown, thick] (x3) -- (y2);
			\draw[teal, thick] (x3) -- (y3);
			
			\draw[lime, thick] (y1) -- (z1);
			\draw[violet, thick] (y1) -- (z2);
			\draw[pink, thick] (y1) -- (z3);
			\draw[olive, thick] (y2) -- (z1);
			\draw[blue!70!black, thick] (y2) -- (z2);   
			\draw[red!70!black, thick] (y2) -- (z3);     
			\draw[blue!50!white, thick] (y3) -- (z1);   
			\draw[yellow!80!black, thick] (y3) -- (z2);  
			\draw[gray!70, thick] (y3) -- (z3);         
			
		\end{tikzpicture}
		
		\caption{$G_2(n,k)$}
		\label{fig2}
	\end{figure}
    
		\item $G_3(n.k)$: Let $V(G_3(n,k))=X\cup Y$ with $\arrowvert X\arrowvert=2k-3$ and $\arrowvert Y\arrowvert=n-2k+3$. 
		We color all edges in $|Y|K_1 \vee K_{|X|} $ with distinct colors, and color $E(K_n)\setminus E(|Y|K_1 \vee K_{|X|} )$ with another new color, depicted in Figure \ref{fig3}, so
		\[c(G_3(n,k)=\left(\begin{matrix} 2k-3 \\ 2\end{matrix}\right)+(n-2k+3)(2k-3)+1.\]

        \begin{figure}
		\centering
		\begin{tikzpicture}[
		 	dot/.style={circle, draw, fill=white, inner sep=0pt, minimum size=3pt},
		 	myellipse/.style={draw, ellipse, minimum width=1.5cm, minimum height=3cm, thick}
		 	]
		 	
		 	\filldraw[fill=gray!30, draw=black, thick] (0,0) ellipse (0.9cm and 1.5cm);
		 	\node at (0,2.0) {$X$};
		 	\node[dot] (x1) at (0,1.0) {};
		 	\node[dot] (x2) at (0,0) {};
		 	\node[dot] (x3) at (0,-1.0) {};
		 	\node at (0,-0.5) {$\vdots$};
		 	\node at (0,0.5) {$rb K_{2k-3}$};
		 	
		 	\draw[thick] (3,0) ellipse (0.75cm and 1.5cm);
		 	\node at (3,2.0) {$Y$};
		 	\node[dot] (y1) at (3,1.0) {};
		 	\node[dot] (y2) at (3,0) {};
		 	\node[dot] (y3) at (3,-1.0) {};
		 	\node at (3,-0.5) {$\vdots$};
		 	
		 	\draw[red, thick] (x1) -- (y1);
		 	\draw[blue, thick] (x1) -- (y2);
		 	\draw[green, thick] (x1) -- (y3);
		 	\draw[orange, thick] (x2) -- (y1);
		 	\draw[purple, thick] (x2) -- (y2);
		 	\draw[cyan, thick] (x2) -- (y3);
		 	\draw[magenta, thick] (x3) -- (y1);
		 	\draw[brown, thick] (x3) -- (y2);
		 	\draw[teal, thick] (x3) -- (y3);
		 	
		\end{tikzpicture}

		\caption{$G_3(n,k)$}
		\label{fig3}
	\end{figure}

\item $G_4(n,k)$: Let $V(G_4(n,k))=X\cup Y$ with $\arrowvert X\arrowvert=3k-1$ and $\arrowvert Y\arrowvert=n-3k+1$. We color all edges in $K_{|X|}$ with distinct colors. Only one color is added for each vertex in $Y$ added, depicted in Figure \ref{fig4}, so
		\[c(G_4(n,k)=\left(\begin{matrix} 3k-1 \\ 2\end{matrix}\right)+n-3k+1.\]

        \begin{figure}
		\centering
		\begin{tikzpicture}[
			dot/.style={circle, draw, fill=white, inner sep=0pt, minimum size=3pt},
			myellipse/.style={draw, ellipse, minimum width=1.5cm, minimum height=3cm, thick}
			]
			
			\filldraw[fill=gray!30, draw=black, thick] (0,0) ellipse (0.9cm and 1.5cm);
			\node at (0,2.0) {$X$};
			\node[dot] (x1) at (0,1.0) {};
			\node[dot] (x2) at (0,0) {};
			\node[dot] (x3) at (0,-1.0) {};
			\node at (0,-0.5) {$\vdots$};
			\node at (0,0.5) {$rb K_{3k-1}$};
			
			\draw[thick] (4,-1.5) ellipse (1.5cm and 0.75cm);
			\node at (6,-1.5) {$Y$};
			\node[dot] (y1) at (3,-1.5) {};
			\node[dot] (y2) at (4,-1.5) {};
			\node[dot] (y3) at (5,-1.5) {};
			\node at (4.5,-1.5) {$\cdots$};
			
			\draw[blue, thick] (y1) to[out=150, in=0, looseness=1.0] (x1);
			\draw[blue, thick] (y1) to[out=160, in=0, looseness=0.9] (x2);
			\draw[blue, thick] (y1) to[out=170, in=0, looseness=1.0] (x3);
			
			\draw[red, thick] (y2) to[out=140, in=10, looseness=1.2] (x1);
			\draw[red, thick] (y2) to[out=160, in=0, looseness=1.0] (x2);
			\draw[red, thick] (y2) to[out=180, in=15, looseness=1.0] (x3);
			\draw[red, thick] (y2) to[out=180, in=0, looseness=1.5] (y1);
			
			\draw[green, thick] (y3) to[out=130, in=30, looseness=1.2] (x1);
			\draw[green, thick] (y3) to[out=150, in=10, looseness=1.0] (x2);
			\draw[green, thick] (y3) to[out=160, in=10, looseness=1.2] (x3);
			
			\draw[green, thick] (y3) to[out=120, in=60, looseness=1.2] (y1);
			
			\draw[green, thick] (y3) to[out=110, in=70, looseness=1.0] (y2);

		\end{tikzpicture}

		\caption{$G_4(n,k)$}
		\label{fig4}
	\end{figure}
    
\end{itemize}

\noindent
{\bf {Proof of Theorem \ref{thm:1}.}}
For three sets $A, B, C$, a triangle $uvwu$ is said to be of type $ABC$ if $u \in A$, $v \in B$, and $w \in C$.
    
For $G_1(n,k)$, any triangles with all edges in $E(K_{|X|} \vee K_{|Y_1|,|Y_2|})$ are of type $XXX$, $XXY_1$, $XXY_2$, or $XY_1Y_2$, so every triangle has at least one vertex in $X$. There are rainbow $(k-2)C_3$ in $(V(K_n), E(K_{|X|} \vee K_{|Y_1|,|Y_2|}))$ because each edge is colored distinctly. In $(V(K_n),E(K_n)\setminus E(K_{|X|} \vee K_{|Y_1|,|Y_2|}))$, one rainbow $C_3$ can be added to rainbow $(k-2)C_3$ to form rainbow $(k-1)C_3$ in $K_n$. This is because all edges in $E(K_n)\setminus E(K_{|X|} \vee K_{|Y_1|,|Y_2|})$ are colored with the left one new color.
	
For $G_2(n,k)$, all triangles with all edges in $E(|Y_1|K_1 \vee (K_{|X|} \cup |Y_2|K_1))$ are of type $XXX$, or $XXY_1$. Thus, every triangle intersects $X$ at least twice, and $(V(K_n), E(|Y_1|K_1 \vee (K_{|X|} \cup |Y_2|K_1)))$ contains rainbow $(k-2)C_3$ because each edge is colored distinctly. In $(V(K_n),E(K_n)\setminus E(|Y_1|K_1 \vee (K_{|X|} \cup |Y_2|K_1)))$, one rainbow $C_3$ can be added to rainbow $(k-2)C_3$'s to form rainbow $(k-1)C_3$ in $K_n$ for the same reason.
	
Similarly, $G_3(n,k)$ has rainbow $(k-1)C_3$ in it.
The reason why $G_4(n,k)$ has rainbow $(k-1)C_3$ but no rainbow $kC_3$ has already been given in Section 3 of \cite{WZLX2023}.
	
When $n$ range from $3k$ to $\frac{-13k^2+25k-8}{8-4k}\approx 3.25k$, $c(G_4(n,k))$ is the most; from about $\frac{-13k^2+25k-8}{8-4k}\approx 3.25k$ to about $4k-6$, $c(G_3(n,k))$ has the maximum value. Then $c(G_3(n,k))$ is larger than $c(G_2(n,k))$ and attains the maximum value until $n=\frac{9k^2-6k}{2k-4}\approx 4.5k$. Finally, $c(G_1(n,k))$ attains the maximum value from $\frac{9k^2-6k}{2k-4}\approx 4.5k$ to infinity. The values of the four functions $c(G_i(n,k))$ are potted in Figure \ref{fig5}, and the thresholds are listed in Table \ref{tab1}.

	\begin{figure}
		\centering
		\includegraphics[width=1\linewidth]{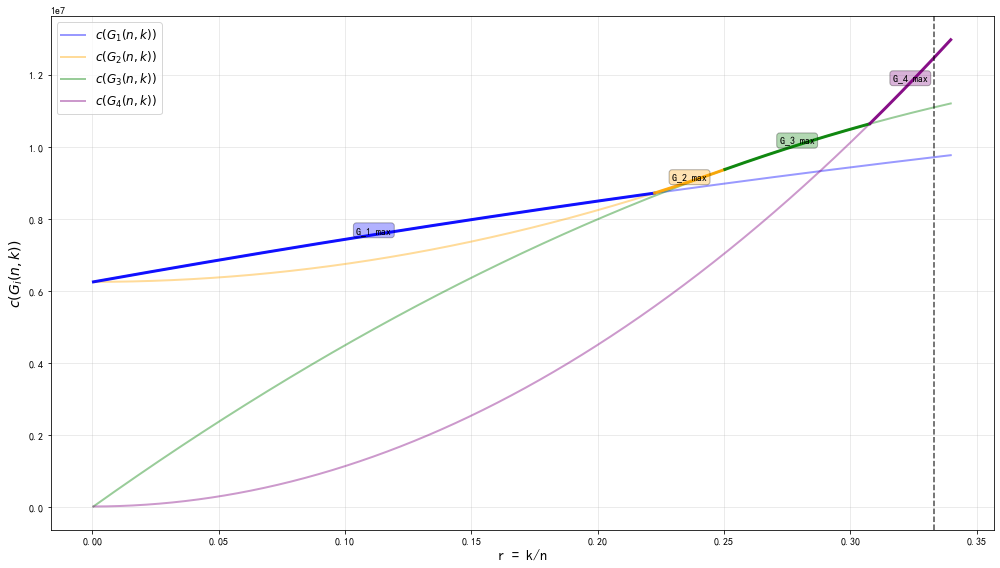}
		\caption{$c(G)$ of the graphs $G_i(n,k)$ where $n$ ranges $3k$ to infinity.}
		\label{fig5}
	\end{figure}
	We can see that $c(G_2(n.k))$ and $c(G_3(n.k))$ are greater than the values in Conjecture \ref{Conj:1} when $n$ is approximately within the range of about $\frac{-13k^2+25k-8}{8-4k}$ to about $\frac{9k^2-6k}{2k-4}$, thus completing the proof of Theorem \ref{thm:1}.
\qed
\begin{table}[h]
	\centering
	\caption{Transitions between $G_i(n,k)$.}
	\begin{tabular}{l >{$}l<{$}}
    \toprule
    \textbf{graph} & \textbf{the range of $n$} \\
    \midrule
    $G_{4}(n,k)$ & 3k \leq n \leq \dfrac{-13k^2+25k-8}{8-4k}\approx 3.25k \\[8pt]
    $G_{3}(n,k)$ & \dfrac{-13k^2+25k-8}{8-4k}\approx 3.25k\  \leq n \leq 4k-6 \\[8pt]
    $G_{2}(n,k)$ & 4k-6 \leq n \leq \dfrac{9k^2-6k}{2k-4}\approx 4.5k  \\[8pt]
    $G_{1}(n,k)$ & n \geq \dfrac{9k^2-6k}{2k-4}\approx 4.5k    \\
    \bottomrule
    \end{tabular}
    \label{tab1}
\end{table}

Motivated by the extremal graphs for vertex-disjoint triangles and Theorem 6 in \cite{ABHP15}, we dare to pose the following conjecture.
\begin{conj}\label{conj2}
		There exists $n_0$ such that for each $n>n_0$ and each $k$, $n \geq 3k$, we have $ar(n,kC_3) = \max_{j \in [4]} c(G_i(n,k))$.
	\end{conj}


\section{Proof of Theorem \ref{Thm: vertexdisj-rainbow-triangles}}
\begin{theorem}\label{Thm: vertexdisj-rainbow-triangles}
Let $n,k$ be two positive integers. Let $G$ be an edge-colored graph of order $n$, where $n\ge 42.5k+48$. If $\delta^c(G)\ge (n+k)/2$ with $\delta^c(G)\ge (n+k)/2$, then $G$ contains $k$ vertex-disjoint rainbow triangles.
\end{theorem}

\begin{lemma}[\cite{LNSZ24}]\label{Lem: delta^c number of rainbow K3}

Let $G$ be an edge-colored graph on $n$ vertices and $e(G)$ edges, $rt(G)$ the number of rainbow $C_3$ of $G$. Suppose that $\delta^c(G)\geq \frac{n+1}{2}$ and $e(G)$ is minimal subject to $\delta^c(G)$, then

\begin{align*}
    rt(G)\ge \frac{e(G)(2\delta^c(G)-n)}{3}\ge\frac{\delta^c(G)(2\delta^c(G)-n)n}{6}.
\end{align*}
\end{lemma}

\begin{lemma}\label{Lem: delta^c > (n+k)/2}
Let $G$ be an edge-colored graph on $n$ vertices and $e(G)$ edges, $rt(G)$ be the number of rainbow $C_3$ of $G$. Suppose that $\delta^c(G)\ge\frac{n+k}{2}$. Then
\begin{align*}
    rt(G)\ge \frac{kn(n+k)}{12}.
\end{align*}    
\end{lemma}
\begin{proof}
We remove edges of $G$ as much as possible so that the resulting spanning subgraph $G'$ satisfies $\delta^c(G')\geq \frac{n+k}{2}$. By Lemma \ref{Lem: delta^c number of rainbow K3},
$rt(G')\geq \frac{\delta^c(G')(2\delta^c(G')-n)n}{6}\geq \frac{kn(n+k)}{12}.$
Observe that $rt(G)\geq rt(G')$. This proves the lemma.
\end{proof}

\noindent\textbf{Proof of Theorem \ref{Thm: vertexdisj-rainbow-triangles}.} 
We prove the theorem by induction on $k$. When $k=1$ and $k=2$, it is reduced to Theorem \ref{Thm:Li} and Hu-Li-Yang's theorem, respectively. Now assume $k\geq 3$ and suppose that the theorem holds for the case $k-1$. We assume that $G$ contains $k-1$ vertex-disjoint rainbow triangles, but no $k$
vertex-disjoint rainbow triangles. We denote by $V_0$ the vertex set of such $k-1$ vertex-disjoint rainbow triangles. 

Now we define some notation used below.
Set $V_0=\bigcup_{i=1}^{k-1}T_i$, where $T_i=\{u_{i,1},u_{i,2},u_{i,3}\}$ is the vertex set of the $i$-th rainbow triangle. Let $V_1=V(G)\setminus V_0$. For $v\in V(G),e\in E(G)$, let $rt(v)$ (resp. $rt(e)$) denote the number of rainbow triangles that contain $v$ (resp. $e$). Similarly, for $v\in V_0,e\in E(G[V_0])$, let $rt_1(v)$ be the number of rainbow $C_3$'s which contain $v$ and two vertices in $V_1$, and $rt_2(e)$ the number of rainbow $C_3$'s which contain $e$ and one vertex in $V_1$. For $s\in\mathbb{N}^+$ and $v\in V(G)$, let $RF_s(v)$ be an edge-colored friendship graph that consists of $s$ triangles with exactly one common vertex $v$, in which each triangle is rainbow.

\begin{claim}\label{Claim: rt(v)>=(s-1)(n-1) RF_s(v)}
Let $v\in V_0$ and $s\in\mathbb{N}^+$. 
\begin{enumerate}
  \item If $rt(v)\ge (s-1)(n-1)+1$, then $G$ contains an $RF_s(v)$.
  \item If $rt_1(v)\ge (s-1)(n-3k+2)$, then $G$ contains an $RF_s(v)$ with $RF_s(v)\cap V_0=\{v\}$.
\end{enumerate}
\end{claim}

\noindent\textbf{Proof of Claim \ref{Claim: rt(v)>=(s-1)(n-1) RF_s(v)}.} Denote by $E_v=\{uw:u,w\in N_G(v),\mbox{vuwv~is~a~rainbow~triangle}\}$. So, for each $e\in E_v$,  there is a rainbow triangle containing both $v$ and $e$.
If $rt(v)\ge (s-1)(n-1)+1$, 
then $|E_v|\geq (s-1)(n-1)+1$. 
Consider the subgraph $G[E_v]$. We have $v(G[E_v])\leq d_G(v)\leq n-1$. Recall that the Erd\H{o}s-Gallai Theorem on paths states that any graph on $n$ vertices and $m$ edges contains a path of length
at least $\frac{2m}{n}$. Thus, there is a path of length at least $\frac{2((s-1)(n-1)+1)}{n-1}>2s-2$, and therefore of length at least $2s-1$, in $G[E_v]$. Since any path with $t$ edges contains a matching of size $\lceil t/2\rceil$, the graph $G[E_v]$ admits a matching of size at least $s$. This matching together with $v$ yields $RF_s(v)$.

Similarly, if $rt_1(v)\ge (s-1)(n-3k+2)$, then $G_v^1=(V_1, E_v)$ contains a matching of size $s$. These $s$ matching edges together with $v$ produce $RF_s(v)$ intersecting $V_0$ only in $v$. $\hfill\square$

By induction, $G$ has the following property.

\begin{claim}(friendship subgraph)\label{Claim:rt<(3k-3)(n-1)}
For any integer $s\geq 3k-2$ and each $v\in V(G)$, $G$ is $RF_s(v)$-free. 
Moreover, we have $rt(v)\leq (3k-3)(n-1)$ and $rt_1(v)\leq (3k-3)(n-3k+2).$
\end{claim}

\noindent\textbf{Proof of Claim \ref{Claim:rt<(3k-3)(n-1)}.} Assume that there exists a vertex $v\in V(G)$ and an integer $s\ge 3k-2$ such that $RF_s(v)\subseteq G$. 
Let $G_v=G-v$. Observe that removing the vertex $v$ decreases the color degree of any other vertex by at most~1, hence $\delta^c(G_v)\ge\delta^c(G)-1\ge\frac{n+k}{2}-1=\frac{(n-1)+(k-1)}{2}$.
By the induction hypothesis, $G_v$ contains $k-1$ vertex-disjoint rainbow triangles. Denote them by $T_v^j$ ($j=1,...,k-1$), which contains $3k-3$ vertices.
Since $RF_s(v)$ consists of $s$ rainbow triangles containing $v$ and $s \geq 3k-2$, at least one of these rainbow triangles must be vertex-disjoint from each $T_v^j$ ($j\in[k-1]$).
So, there are $k$ vertex-disjoint rainbow triangles in $G$, a contradiction. 
The other conclusion follows from Claim \ref{Claim: rt(v)>=(s-1)(n-1) RF_s(v)}.
$\hfill\square$

It is easy to see that each rainbow triangle must contain at least one vertex of $V_0$.
Since $\binom{3k-3}{3}<\frac{kn(n+k)}{12}$ when $n^2\ge 54k^2$, there exists a rainbow $C_3$ that contains at least one vertex of $V_1$. 

For $t\in [k-1]$ and $\{u_1,...,u_{k-1}\}\subseteq V_0$, similar to the definition of $RF_s(v)$, we denote by $RF_{2,1}(t;u_1,...,u_{k-1})$ a graph which is the vertex-disjoint union of $RF_2(u_i)$ (the so called hourglass) ($i\in[t]$) and $RF_1(u_j)$ (the rainbow triangle) ($j \in [k-1]\setminus[t] $), as illustrated in Figure \ref{Figure: RF_{2,1}(t;v_1,...,v_{k-1})}. Let $t$ be the maximum integer in $[k-1]$ such that $G$ contains a $RF_{2,1}(t;u_1,...,u_{k-1})$. Then, we relabel all vertices such that $V_0=\bigcup_{i=1}^{k-1}T_i=\bigcup_{i=1}^{k-1}\{v_{i,1},v_{i,2},v_{i,3}\}$, where $T_1,...,T_{k-1}$ are $k-1$ vertex-disjoint rainbow $C_3$, $v_{i,1}$ is the common vertex of two triangles of $T_i$, and $v_{i,4}, v_{i,5}$ are the remaining two vertices in $V(RF_2(v_i))-V(T_i)$ ($i\in[t]$).

\begin{figure}[H]
    \begin{center}
        \begin{tikzpicture}[node distance=2cm, auto]

    \coordinate [label=above:{$v_{1,2}$}](v12) at (0.8,4);\fill (v12) circle (2pt);
    \coordinate [label=above:{$v_{1,3}$}](v13) at (2.2,4);\fill (v13) circle (2pt);
        \coordinate [label=left:{$v_{1,1}$}](v11) at (1.5,3);\fill  (v11) circle (2pt);
    \coordinate [label=below:{$v_{1,4}$}](v14) at (0.8,2);\fill (v14) circle (2pt);
    \coordinate [label=below:{$v_{1,5}$}](v15) at (2.2,2);\fill (v14) circle (2pt);

    \coordinate [label=right:{}](d1) at (3,3.5);\draw (d1) circle (1pt);
    \coordinate [label=right:{}](d2) at (3.5,3.5);\draw (d2) circle (1pt);
    \coordinate [label=right:{}](d3) at (4,3.5);\draw (d3) circle (1pt);

    \coordinate [label=above:{$v_{t,2}$}](vt2) at (4.8,4);\fill (vt2) circle (2pt);
    \coordinate [label=above:{$v_{t,3}$}](vt3) at (6.2,4);\fill (vt3) circle (2pt);
        \coordinate [label=left:{$v_{t,1}$}](vt1) at (5.5,3);\fill  (vt1) circle (2pt);
    \coordinate [label=below:{$v_{t,4}$}](vt4) at (4.8,2);\fill (vt4) circle (2pt);
    \coordinate [label=below:{$v_{t,5}$}](vt5) at (6.2,2);\fill (vt5) circle (2pt);

    \coordinate [label=above:{$v_{t+1,2}$}](2vt+1) at (7.5,4);\fill (2vt+1) circle (2pt);
    \coordinate [label=above:{$v_{t+1,3}$}](3vt+1) at (8.9,4);\fill (3vt+1) circle (2pt);
        \coordinate [label=below:{$v_{t+1,1}$}](1vt+1) at (8.2,3);\fill  (1vt+1) circle (2pt);

    \coordinate [label=right:{}](d4) at (9.7,3.5);\draw (d4) circle (1pt);
    \coordinate [label=right:{}](d5) at (10.2,3.5);\draw (d5) circle (1pt);
    \coordinate [label=right:{}](d6) at (10.7,3.5);\draw (d6) circle (1pt);

    \coordinate [label=above:{$v_{k-1,2}$}](2vk-1) at (11.5,4);\fill (2vk-1) circle (2pt);
    \coordinate [label=above:{$v_{k-1,3}$}](3vk-1) at (12.9,4);\fill (3vk-1) circle (2pt);
        \coordinate [label=below:{$v_{k-1,1}$}](1vk-1) at (12.2,3);\fill  (1vk-1) circle (2pt);

    \draw (v12)--(v13)--(v11)--(v12);\draw (v11)--(v14)--(v15)--(v11);
    \draw (vt2)--(vt3)--(vt1)--(vt2);\draw (vt4)--(vt5)--(vt1)--(vt4); 
    \draw (2vt+1)--(1vt+1)--(3vt+1)--(2vt+1);
    \draw (2vk-1)--(1vk-1)--(3vk-1)--(2vk-1); 

\end{tikzpicture}
    \end{center}
    \centering
    \caption{$RF_{2,1}(t;v_{1,1},...,v_{k-1,1})$}
    \label{Figure: RF_{2,1}(t;v_1,...,v_{k-1})}
\end{figure}
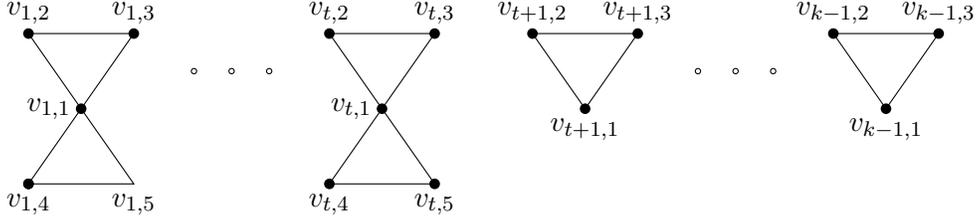
Our main aim is to count $rt(G)$. The main idea is to show that 
$rt(G)<\frac{kn(n+k)}{12}$, which will contradict Lemma \ref{Lem: delta^c > (n+k)/2}. 
To do so, we divide all rainbow triangles into three classes: 
the first class, denoted by $RT_{1}$, consists of those containing exactly two vertices from $V_1$; 
the second class, denoted by $RT_{2}$, consists of those containing at most one vertex of $V_1$, hence each must contain an edge of $G[V_0]$,
and the third class, denoted by $RT_{3}$, consists of those entirely contained in $V_0$.


We use $$E_{1}=\bigcup_{i\in [t],j\in[k-1]\setminus\{i\}}E[v_{i,1},V(T_j)]$$ to denote the subset of edges in $G[V_0]$, in which one end-point of each edge is a center vertex from the first $k$ triangles and the other end-point is from some triangle other than the one contains the center vertex.

We first obtain an upper bound of $|RT_{1}|$. 

\begin{claim}\label{Claim: rt_1(u1)>=,rt_1(u2)<=}
For any $i\in[k-1]$, let $u_1,u_2\in V(T_i)$. If $rt_1(u_1)>n-3k+2$ then $rt_1(u_2)\le4$. Moreover, if $1\le rt_1(u_1)\le n-3k+2$ then $rt_1(u_2)\le 2(n-3k+2)$.
\end{claim} 
\noindent\textbf{Proof of Claim \ref{Claim: rt_1(u1)>=,rt_1(u2)<=}.} As $rt_1(u_1)> n-3k+2$, by Claim \ref{Claim: rt(v)>=(s-1)(n-1) RF_s(v)}, $G$ contains a  $RF_2(u_1)$ with the unique vertex $u_1$ in $V_0$. 
Suppose $V(RF_2(u_1))=\{u_1,u_1^1,u_1^2,u_1^3,u_1^4\}$, where both $G[\{u_1,u_1^1,u_1^2\}]$ and $G[\{u_1,u_1^3,u_1^4\}]$ are rainbow $C_3$'s. 
Then, for each rainbow $C_3$, named $T_{u_2}$,
which contains $u_2$ and two vertices in $V_1$, we infer $V(T_{u_2})\cap \{u_1^1,u_1^2\}\neq\emptyset$ and $V(T_{u_2})\cap \{u_1^3,u_1^4\}\neq\emptyset$; otherwise if $V(T_{u_2})\cap \{u_1^1,u_1^2\}=\emptyset$, then $T_{iu_2},G[\{u_1,u_1^1,u_1^2\}],T_1,...,T_{i-1},T_{i+1},...,T_{k-1}$ are $k$ vertex-disjoint rainbow $C_3$'s, a contradiction. Thus, $rt_1(u_2)\le\binom{2}{1}\cdot\binom{2}{1}=4$.

Assume that $rt_1(u_1)\le n-3k+2$. By Claim \ref{Claim: rt(v)>=(s-1)(n-1) RF_s(v)} $(s=3)$, if $rt_1(u_2)>2(n-3k+2)$, then $G$ contains an $RF(u_2)$ in which $u_2$ is the unique vertex of $V_0$. Suppose $e_0=x_1x_2\in E(G[V_1])$ together with $u_1$ forms a rainbow $C_3$. Denote by $e_1,e_2,e_3\in E(RF_3^3(u_2))\bigcap E(G[V_1])$ the disjoint edges in which each edge together with $u_2$ forms a rainbow $C_3$. Then at least one of them contains no $x_1$ and $x_2$. Suppose such an edge is $e_1=y_1y_2$. Then, $G[\{u_1,x_1,x_2\}],G[\{u_2,y_1,y_2\}],T_1,...,T_{i-1},T_{i+1},...,T_{k-1}$ are $k$ vertex-disjoint rainbow $C_3$'s. Thus, $rt_1(u_2)\le 2(n-3k+2)$. $\hfill\square$ 

For $i\in[t]$, we estimate the number of rainbow $C_3$'s containing $v_{i,1}$ and some vertices in $V_1$. Denote by $E_{1,i}=\{e\in E_1:e$ is incident to $v_{i,1}\}$. It is easy to see $\bigcup_{i=1}^t E_{1,i}=E_1$. 
For a rainbow triangle, denote by $T(v_{i,1})$, which contains $v_{i,1}$ and some vertices in $V_1$ (maybe one or two), if it contains two vertices in $V_1$, then it contributes to $rt_1(v_{i,1})$; if it contains two vertices in $V_0$, then it must contain an edge $e\in E_{1,i}$.

We now try to bound $rt_1(v_{i,1})$ by dealing with three cases:
When $1\le rt_1(v_{i,1})\le4$, by Claim \ref{Claim: rt_1(u1)>=,rt_1(u2)<=}, we have the following. 
If $n-3k+2< rt_1(v_{i,2})\le 2(n-3k+2)$, then $rt_1(v_{i,3})\le 4$, and hence $rt_1(v_{i,2})+rt_1(v_{i,3})\le 2(n-3k+2)+4$. 
Otherwise, if $rt_1(v_{i,2})\le n-3k+2$, then 
$rt_1(v_{i,2})+rt_1(v_{i,3})\le 2(n-3k+2)$. 
Then, by Claim \ref{Claim:rt<(3k-3)(n-1)}, we obtain

\begin{align*}
    \sum_{j=1}^3 rt_1(v_{i,j})+\sum_{e\in E_{1,i}}rt_2(e)
    &\le rt_1(v_{i,2})+rt_1(v_{i,3})+rt(v_{i,1})\\
    &\le 3(n-3k+2)+ (3k-3)(n-1)\\
    &=3kn-12k+9.
\end{align*}

When $4<rt_1(v_{i,1})\le n-3k+2$, by Claim \ref{Claim: rt_1(u1)>=,rt_1(u2)<=}, we have $rt_1(v_{i,2})\le n-3k+2$ and $rt_1(v_{i,3})\le n-3k+2$. 
Here we note that if $rt_1(v_{i,2}) > n-3k+2$ or $rt_1(v_{i,3}) > n-3k+2$, then $rt_1(v_{i,1})\le 4$, which is a contradiction. 

By Claim \ref{Claim:rt<(3k-3)(n-1)}, we have
\begin{align*}
    \sum_{j=1}^3 rt_1(v_{i,j})+\sum_{e\in E_{1,i}}rt_2(e)
    &\le rt_1(v_{i,2})+rt_1(v_{i,3})+rt(v_{i,1})\\
    &\le (n-3k+2)+(n-3k+2)+(3k-3)(n-1)\\
    &=(3k-1)(n-1)-6k+6\\
    &=(3k-1)n-9k+7.
\end{align*}

When $rt_1(v_{i,1})> n-3k+2$, by Claim \ref{Claim: rt_1(u1)>=,rt_1(u2)<=}, $rt_1(v_{i,2})\le4$
and $rt_1(v_{i,3})\le4$. Then, by Claim \ref{Claim:rt<(3k-3)(n-1)}, we have
\begin{align*}
    \sum_{j=1}^3 rt_1(v_{i,j})+\sum_{e\in E_{1,i}}rt_2(e)
    &\le rt_1(v_{i,2})+rt_1(v_{i,3})+rt(v_{i,1})\\
    &\le 4+4+(3k-3)(n-1)\\
    &\le (3k-3)(n-1)+8\\
    &=(3k-3)n-3k+11.
\end{align*}

By the above three inequalities, we have

\begin{claim}\label{Claim: 4}
    $\sum_{i=1}^t\sum_{j=1}^3 rt_1(v_{i,j})+\sum_{e\in E_1}rt_2(e)\le t(3kn-12k+9).$
\end{claim}

\noindent
\textbf{Proof of Claim \ref{Claim: 4}.} For each $i\in[t]$, since $n \ge 42.5k+48$, we know $$\sum_{j=1}^3 rt_1(v_{i,j})+\sum_{e\in E_{1,i}}rt_2(e)\le 3kn-12k+9,$$ and so 

\begin{align*}
    \sum_{i=1}^t\sum_{j=1}^3 rt_1(v_{i,j})+\sum_{e\in E_1}rt_2(e)&\le \sum_{i=1}^t\left(\sum_{j=1}^3 rt_1(v_{i,j})+\sum_{e\in E_{1,i}}rt_2(e)\right)\\
    &\le t(3kn-12k+9).
\end{align*}
$\hfill\square$

Now we estimate $rt_1(v_{i,j})$ for $i\ge t+1$ and $j\in[3]$. 

\begin{claim}\label{Claim: rt1(vi),i>=t+1}
For any $i\in[k-1] \setminus [t]$, we have $$rt_1(v_{i,1})+rt_1(v_{i,2})+rt_1(v_{i,3})\le \max\{2t(n-3k+2)+8,3(n-3k+2)\},$$ and $$\sum_{i=t+1}^{k-1}\sum_{j=1}^3rt_1(v_{i,j})\le \max\{(k-1-t)(2t(n-3k+2)+8),3(k-1-t)(n-3k+2)\}.$$
\end{claim}

\noindent\textbf{Proof of Claim \ref{Claim: rt1(vi),i>=t+1}.} 
For each $j\in[3]$, let $T(v_{i,j})$ be a rainbow $C_3$ containing $v_{i,j}$ and two vertices of $V_1$. 
We claim that \[V(T(v_{i,j}))\bigcap(\bigcup_{s=1}^t\{v_{s,4},v_{s,5}\})\neq \emptyset;\]
Otherwise, we obtain a $RF_{2,1}(t+1;v_{1,1},...,v_{t+1,j},...)$, contradicting the maximality of $t$.
There are at most $2t$ vertices in the set $\bigcup_{s=1}^t\{v_{s,4},v_{s,5}\}$. 
The maximum number of rainbow $C_3$'s which contain some vertices in $\bigcup_{s=1}^{t}\{v_{s,4},v_{s,5}\}$ is at most $2t(n-3k+2)$. Thus, $rt_1(v_{i,j})< 2t(n-3k+2).$ 

If there exists some $j\in[3]$ (assume $j=1$) such that $rt_1(v_{i,1})> n-3k+2$, then by Claim \ref{Claim: rt_1(u1)>=,rt_1(u2)<=}, we have $$\sum_{j=1}^3 rt_1(v_{i,j})\le 2t(n-3k+2)+4+4;$$ 
if $rt_1(v_{i,j})<n-3k+2$ for each $j\in[3]$, then $\sum_{j=1}^3 rt_1(v_{i,j})<3(n-3k+2)$.  Thus, $$\sum_{i=t+1}^{k-1}\sum_{j=1}^3 rt_1(v_{i,j})\le \max\{(k-1-t)(2t(n-3k+2)+8),(k-1-t)3(n-3k+2)\}.$$ $\hfill\square$

Next, we obtain an upper bound of $|RT_2|$. For each $e\in E(G[V_0])$ and $v\in V_1$, we denote by $RT(e,v)$ the rainbow $C_3$ that contains $e$ and $v$. To estimate $|RT_2|$, we divide the edges in $G[V_0]$ into five classes. 

Recall $$E_{1}=\bigcup_{i\in [t],j\in[k-1]\setminus\{i\}}E[v_{i,1},V(T_j)].$$ 
We use
$$E_2=\bigcup_{i\in[t],j\in[k-1]\setminus[t]}E[V(T_i)-v_{i,1},V(T_j)]$$
to denote the edges joining one non center-vertex from the first $t$ rainbow triangles and one vertex from last $k-t-1$ triangles;
$$E_3=\bigcup_{t+1\le i<j\le k-1}E[V(T_i),V(T_j)]$$
to denote all edges between any two triangles from the last $k-t-1$ vertex disjoint rainbow triangles;
$$E_4=\bigcup_{i=1}^{k-1}E(T_i)$$ 
to denote the edges within all these $k-1$ vertex disjoint rainbow triangles;
and $$E_5=\bigcup_{1\le i<j\le t}E[V(T_i)-v_{i,1},V(T_j)-v_{j,1}]$$
to denote the edges between all non-center vertices of the first $t$ vertex-disjoint rainbow triangles.
Obviously, $E_i$'s are pairwise disjoint, and $E(G[V_0])=\bigcup_{i=1}^5 E_i$.

\begin{claim}\label{Claim: rt(e) E3}
    $\sum_{e\in E_3}rt_2(e)\le 6\binom{k-1-t}{2}(n-3k+3)$.
\end{claim}

\noindent\textbf{Proof of Claim \ref{Claim: rt(e) E3}.} For each $t_1,t_2\in[k-1]$ $(t_1< t_2)$, denote by $e_j$ ($j=1,2,3$) three disjoint edges which have end-vertices between $V(T_{t_1})$ and $V(T_{t_2})$. If two of the three edges, say $e_1$ and $e_2$, satisfy that $rt_2(e_1)\ge3$ and $rt_2(e_2)\ge2$, suppose that $G$ contains $RT(e_1,x_i)$ and $RT(e_2,y_j)$ with $i=1,2,3$ and $j=1,2$, then $rt_2(e_3)=0$. 
Otherwise, assume $G$ contains $RT(e_3,z)$. Since $rt_2(e_1)\ge 3$ and $rt_2(e_2)\ge 2$, we may choose $x_{i_0},y_{j_0},z$ to be distinct. Hence, $RT(e_1,x_{i_0}),RT(e_2,y_{j_0}),RT(e_3,z)$ together with $T_1,\dots,T_{t_1-1},T_{t_1+1},\dots,T_{t_2-1},T_{t_2+1},\dots,T_{k-1}$ form $k$ vertex-disjoint rainbow $C_3$'s, a contradiction.
Thus, $$\sum_{j=1}^3rt_2(e_j)\le \max\{2(n-3k+3),3\cdot 2\} = 2(n-3k+3).$$ 
The nine edges in $G[V(T_{t_1}),V(T_{t_2})]$ can be partitioned into three sets, in which each consists of three pairwise disjoint edges. Denote by $E_3^{i,j}$ the edges in $E_3$ whose end vertices lie in $T_{i}$ and $T_j$. Then for $t+1\le i<j\le k-1$, we have $\sum_{e\in E_3^{i,j}}rt_2(e)\le 3\cdot 2(n-3k+3)$, thus 
\begin{align*}
    \sum_{e\in E_3}rt_2(e)&=\sum_{t+1\le i<j\le k-1}\sum_{e\in E_3^{i,j}}rt_2(e)\\
    &\le\binom{k-1-t}{2}\cdot 6(n-3k+3).
\end{align*}
$\hfill\square$

\begin{claim}\label{Claim: E2}
    $\sum_{e\in E_2}rt_2(e)\le t(k-1-t)(3(n-3k+3)+9).$
\end{claim}
\noindent\textbf{Proof of Claim \ref{Claim: E2}.} For $e\in E_2$, let $t_1\in [t]$ and $t_2\in [k-1]\setminus[t]$. 
Denote by $e_1=u_{1,1}u_{1,2}$ and $e_2=u_{2,1}u_{2,2}$ two disjoint edges between $T_{t_1}-v_{t_1,1}$ and $T_{t_2}$. Then, either $rt_2(e_1)\le 3$ or $rt_2(e_2)\le3$; otherwise, suppose $rt_2(e_1)\ge4$ and $rt_2(e_2)\ge 4$, denote by $N_{2,1}(e_j)$ $(j=1,2)$ the set of vertices in $V_1$ in which each together with $e_j$ forms a rainbow $C_3$. Then, we can choose $w_1\in N_{2,1}(e_1) \setminus \{v_{t_1,4},v_{t_1,5}\}$ and $w_2\in N_{2,1}(e_2) \setminus \{v_{t_1,4},v_{t_1,5},w_1\}$, and obtain three disjoint rainbow $C_3$'s: $RT(e_1,w_1), RT(e_2,w_2)$, $G[\{v_{t_1,1},v_{t_1,4},v_{t_1,5}\}]$, which are disjoint with $T_1,...,T_{t_1-1},T_{t_1+1},...,T_{k-1}$, a contradiction. Thus,$$rt_2(e_1)+rt_2(e_2)\le (n-3k+3)+3,$$ 
and $\sum_{e\in E[T_{t_1}-v_{t_1,1},T_{t_2}]}rt_2(e)\le 3((n-3k+3)+3)$. Since $t_1\in[t]$, $t_2\in [k-1] \setminus [t]$, we finish the proof. $\hfill\square$

\begin{claim}\label{Claim E5}
    $\sum_{e\in E_5} rt_2(e)\le 16\cdot\binom{t}{2}$.
\end{claim}
\noindent\textbf{Proof of Claim \ref{Claim E5}.} Let $1\le t_1< t_2\le t$.  For each edge $e=u_1u_2$ with $u_i\in V(T_i)-\{v_{t_i,1}\}$, if $rt_2(e)\ge5$, then there exists $u\in V_1-\{v_{t_1,4},v_{t_1,5},v_{t_2,4},v_{t_2,5}\}$. Replacing $T_{t_1}$ and $T_{t_2}$ by three rainbow triangles $G[\{v_{t_1,1},v_{t_1,4},v_{t_1,5}\}]$,$G[\{v_{t_2,1},v_{t_2,4},v_{t_2,5}\}]$, $RT(e,u)$, we obtain $k$ vertex-disjoint rainbow triangles. Thus, $rt_{2,1}(u_1u_2)\le4$. It follows $$\sum_{\substack{u_1\in V(T_{t_1})-v_{t_1,1}\\u_2\in V(T_{t_2})-v_{t_2,1}}}rt_2(u_1u_2)\le\binom{2}{1}\binom{2}{1}\cdot4.$$ 
Since there are $\binom{t}{2}$ such pairs $(t_1,t_2)$, 
summing over all of them yield
\[\sum_{e\in E_5} rt_2(e) \;\le\; 16\binom{t}{2}.\]
$\hfill\square$

We can now count $rt(G)$ as follows.
\begin{align*}
rt(G)&\le|RT_1|+|RT_2|+|RT_3|\\
     &\le  \sum_{i=1}^{k-1}\sum_{j
     =1}^3 rt_{1}(v_{i,j})
     +  \sum_{e\in E(G[V_0])}rt_{2}(e)
     +  \binom{3k-3}{3}\\
     &\le \{\textcolor{blue}{\sum_{i=1}^{t}\sum_{j=1}^3 rt_{1}(v_{i,j})}
     +  \textcolor{orange}{\sum_{i=t+1}^{k-1}\sum_{j=1}^3 rt_{1}(v_{i,j})}\} \\
     &\quad+ \{\textcolor{blue}{\sum_{e\in E_1}rt_{2}(e)}  
     +   \textcolor{purple}{\sum_{e\in E_2}rt_{2}(e)}  
     +  \textcolor{red}{\sum_{e\in E_3} rt_{2}(e)} 
     +  \sum_{e\in E_4} rt_{2}(e)
     +  \textcolor{green}{\sum_{e\in E_5}rt_{2}(e)}   \}\\ 
     &\quad   
     +   \binom{3k-3}{3}:=A_1.
\end{align*}

By Claim \ref{Claim: 4}, \ref{Claim: E2}, \ref{Claim: rt(e) E3}, \ref{Claim E5} and Claim \ref{Claim: rt1(vi),i>=t+1}, we have an upper bound for the \textcolor{blue}{blue}, \textcolor{purple}{purple}, \textcolor{red}{red}, \textcolor{green}{green} and \textcolor{orange}{orange} term, respectively. There also holds $\sum_{e\in E_4} rt_2(e)\le \sum_{i=1}e(T_i)(n-3k+3)\le(3k-3)(n-3k+3)$. Thus, we have 
\begin{align*}
    A_1\quad &\le \textcolor{orange}{\max\{(k-1-t)(2t(n-3k+2)+8),3(k-1-t)(n-3k+2)\}}   \\
     &\quad +\textcolor{blue}{t(3kn-12k+9)}
     +    \textcolor{purple}{3t(k-1-t)((n-3k+3)+3)}
     +   \textcolor{red}{6\binom{k-1-t}{2}(n-3k+3)}  \\
     &\quad +(3k-3)(n-3k+3) +   \textcolor{green}{16\binom{t}{2}} +   \binom{3k-3}{3}:=A_2.
\end{align*}
For $a,b\in\mathbb{N}$, obviously $\binom{a}{b}\le\frac{a^b}{b!}$. 

If $t\ge2$, then by $(k-1-t)\le(k-t)$, $k \ge 3$ and $t\le k$, we have
\begin{align*}
    A_2\quad&\le t^2(-2n+6k-5)+t(2kn+3k^2-8k+1)\\
    &+(3k^2n+(3k-3)n-\frac{9}{2}k^3-\frac{27}{2}k^2+\frac{79}{2}k-\frac{27}{2})   \\
        &\le -2nt^2+2knt+k^2(6k-5)+k(3k^2-8k+1)+\\
        &\quad \{(3k^2+3k-3)n-\frac{9}{2}k^3-\frac{27}{2}k^2+\frac{79}{2}k\}\\
        &= -2nt^2+2knt+(3k^2+3k-3)n+(6+3-\frac{9}{2})k^3-(5+8+\frac{27}{2})k^2+(1+\frac{79}{2})k\\
        &\le -2nt^2+2knt+(3k^2+3k-3)n+\frac{9}{2}k^3-13k^2-\frac{27}{2}k^2+\frac{27}{2}k^2 ~(\mbox{as}~k\geq 3)\\
        &= -2nt^2+2knt+(3k^2+3k-3)n+\frac{9}{2}k^3-13k^2:=f(t)=B_1.
\end{align*}
We claim
$f(t)$ is maximal when $t=\frac{k}{2}$, and so,
\begin{align*}
    B_1\le \frac{7}{2}k^2n+3kn+\frac{9}{2}k^3-13k^2:=g(k)=B_2.
\end{align*}
Set $$n_0=\frac{41k^2+36k+\sqrt{(41k^2+36k)^2+4k(54k^3-156k^2)}}{2k}.$$ By computing, we have $$n\ge\frac{85}{2}k+48>\frac{41k^2+36k+(44k^2+36k)}{2k}\ge n_0.$$ By Lemma \ref{Lem: delta^c > (n+k)/2}, we have
\begin{align*}
    12\cdot(rt(G)- B_2)&\geq kn^2+k^2n-\{(42k^2+36k)n+54k^3-156k^2\}\\
        &=kn^2+(-41k^2-36k)n-54k^3+156k^2\\
        &\ge kn^2+(-41k^2-36k)n-54k^3+156k^2 ~(\mbox{as}~n > n_0) \\
        &> kn_0^2+(-41k^2-36k)n_0-54k^3+156k^2\\
        &=0,
\end{align*}
which implies $rt(G)>rt(G)$, a contradiction.

If $t\le 1$, then
\begin{align*}
    A_2\quad &\le -t^2+(-3n+9k^2-3k+3)t+3k^2n+(6k-3)n-\frac{9}{2}k^3-\frac{45}{2}k^2+\frac{75}{2}k-\frac{27}{2} \\
    \quad&\le-t^2+(-3n+3)t+9k^2-3k+(3k^2+6k)n-\frac{9}{2}k^3-\frac{45}{2}k^2+\frac{75}{2}k ~(\mbox{as}~t\le 1)\\
    \quad&\le-t^2+(-3n+3)t+9k^2+(3k^2+6k)n-\frac{9}{2}k^3-\frac{45}{2}k^2+\frac{23}{2}k^2 ~(\mbox{as}~k\geq 3)\\
    &=-t^2+(-3n+3)t+(3k^2+6k)n-\frac{9}{2}k^3-2k^2:=h(t)=C_1
\end{align*}
Then, $h(t)$ is maximal when $t=0$, 
\begin{align*}
    C_1 \quad&\le (3k^2+6k)n-\frac{9}{2}k^3-2k^2:=C_2
\end{align*}
Set
$$n_1=\frac{35k^2+72k+\sqrt{(35k^2+72k)^2-4k(54k^3+24k^2)}}{2k}.$$
By computing, we have
$$n\ge \frac{85}{2}k+48>\frac{35k^2+72k+(32k^2+73k)}{2k}> n_1.$$
By Lemma \ref{Lem: delta^c > (n+k)/2}, we have
\begin{align*}
    12(rt(G)-C_2)&\ge kn^2+k^2n-(36k^2+72k)n+54k^3+24k^2\\
    &>kn_1^2-(35k^2+72k)n_1+54k^3+24k^2\\
    &=0,
\end{align*}
a contradiction.
Thus, $G$ contains a rainbow $C_3$ vertex-disjoint with each $T_i$ ($i\in[k-1]$). The proof is complete.
$\hfill\square$




\bibliographystyle{unsrt}

\begin{thebibliography}{99}
\bibitem{ABHP15}
P. Allen, J. B\"{o}ttcher, J. Hladk\'{y}, D. Piguet,
A density Corr\'adi-Hajnal theorem,
\emph{Canad. J. Math.} 67 (2015), no. 4, 721--758.

\bibitem{A83}
N. Alon,
On a conjecture of Erd\H{o}s, Simonovits, and S\'os concerning anti-Ramsey theorems,
\emph{J. Graph Theory} {\bfseries 7} (1983), no. 1, 91--94.

\bibitem{CKRY16}
R. \v{C}ada, A. Kaneko, Z. Ryj\'a\v{c}ek, K. Yoshimoto,Rainbow cycles in edge-colored graphs.     \emph{Discrete Math.} {\bfseries 339} (2016), no. 4, 1387--1392.

\bibitem{CN25}
X. Chen, B. Ning, Rainbow triangles sharing one common vertex or edge, \emph{Electron. J. Combin.} {\bfseries 32} (2025), no. 3, P3.30.

\bibitem{CLN22}
X. Chen, X. Li, B. Ning, Note on rainbow triangles in edge-colored graphs, \emph{Graphs Combin.} {\bfseries 38} (2022), no. 3, Paper No. 69, 13 pp.

\bibitem{CH63}
K. Corr\'adi, A. Hajnal,
On the maximal number of independent circuits in a graph,
\emph{Acta Math. Acad. Sci. Hungar.} {\bfseries 14} (1963), 423--439.

\bibitem{D63}
G.A. Dirac, On the maximal number of independent triangles in graphs, \emph{Abh. Math. Sem. Univ. Hamburg} {\bfseries 26} (1963), 78--82.

\bibitem{ESS65}
P. Erd\H{o}s, M. Simonovits, V.T. Sós (1965), Anti-Ramsey theorems. In Infinite and Finite Sets (Keszthely, Hungary, 1973), Vol. 10 of Coll. Math. Soc. J. Bolyai, pp. 657--665.

\bibitem{FMO10}
S. Fujita, C. Magnant, K. Ozeki, Rainbow generalizations of Ramsey theory: a survey, \emph{Graphs Combin.} {\bfseries 26} (2010), no. 1, 1--30. 


\bibitem{HLY20}
J. Hu, H. Li, D. Yang,
Vertex-disjoint rainbow triangles in edge-colored graphs,
\emph{Discrete Math.} {\bfseries 343} (2020), no. 12, 112117, 5 pp.

\bibitem{JW03}
T. Jiang, D.B. West, On the Erd\H{o}s-Simonovits-S\'os conjecture about the anti-Ramsey number of a cycle. Special issue on Ramsey theory,
\emph{Combin. Probab. Comput.} {\bfseries 12} (2003), no. 5-6, 585--598. 

\bibitem{L13}
H. Li, Rainbow $C_3$'s and $C_4$'s in edge-colored graphs, \emph{Discrete Math.} {\bfseries 313} (2013), no. 19, 1893--1896.

\bibitem{LNXZ14}
B. Li, B. Ning, C. Xu, S. Zhang, Rainbow triangles in edge-colored graphs. \emph{European J. Combin.} {\bfseries 36} (2014), 453--459.

\bibitem{LW12}
H. Li, G. Wang, Color degree and heterochromatic cycles in edge-colored graphs, \emph{European J. Combin.} {\bfseries 33} (2012), no. 8, 1958--1964.

\bibitem{LNSZ24}
X. Li, B. Ning, Y. Shi, S. Zhang, Counting rainbow triangles in edge-colored graphs, \emph{J. Graph Theory} 107 (2024), no. 4, 742--758.

\bibitem{LW-Arxiv-24}
A. Lo, E. Williams, Towards an edge-coloured Corrádi--Hajnal theorem, arXiv:2408.10651.

\bibitem{LLM25+}
H. Lu, X. Luo, X. Ma, New bounds on the Anti-Ramsey number of independent triangles, arXiv:2506.07115v1. 8 Jun 2025.

\bibitem{MN05}
J.J. Montellano-Ballesteros, V. Neuman-Lara, An anti-Ramsey theorem on cycles, \emph{Graphs Combin.} {\bfseries 21} (2005) 343--354.

\bibitem{WZLX2023}
F. Wu, S. Zhang, B. Li, J. Xiao, Anti-Ramsey numbers for vertex-disjoint triangles, \emph{Discrete Math.} {\bfseries 346} (2023), 113123.

\bibitem{YZ19}
L. Yuan, X. Zhang, Anti-Ramsey numbers of graphs with some decomposition family sequences, arXiv:1903.10319. 

\end{thebibliography}

\end{document}